\documentclass{article}
\usepackage{amsmath,amssymb,amsfonts,amsthm,enumerate,xcolor,enumerate,stmaryrd,mathtools}
\usepackage[all]{xy}
\newcommand{\A}{\mathbb{A}}
\newcommand{\B}{\mathbb{B}}

\newcommand{\C}{\mathbb{C}}
\newcommand{\X}{\mathbb{X}}

\NewCommandCopy{\OldS}{\S}
\renewcommand{\S}{\mathbb{S}}
\newcommand{\two}{\textbf{2}}
\newcommand{\op}{\text{op}}
\newcommand{\bS}{\mathbf{S}}

\newcommand{\SE}[1]{\mathbf{SplExt}(#1)}
\newcommand{\FSE}[2]{\mathbf{SplExt}_{#1}(#2)}
\newcommand{\TFSE}[2]{\overline{\mathbf{SplExt}_{#1}(#2)}}

\newcommand{\spexc}[1]{\llbracket#1\rrbracket}
\newcommand{\SplExt}{\textnormal{SplExt}}
\newcommand{\Aut}{\text{Aut}}
\newcommand{\Der}{\text{Der}}
\newcommand{\End}{\text{End}}

\newcommand{\Set}{\mathbf{Set}}

\newtheorem{theorem}{Theorem}[section]
\newtheorem{proposition}[theorem]{Proposition}
\newtheorem{lemma}[theorem]{Lemma}
\newtheorem{corollary}[theorem]{Corollary}
\newtheorem{remark}[theorem]{Remark}

\begin{document}
\title{More on action representability and normalizers}
\author{J. R. A. Gray}
\maketitle
\begin{abstract}
We show that under mild assumptions a functor is a prefibration
if and only if it has terminal objects in its fibers and admits
precartesian liftings of monomorphisms. We use this to show that 
the following conditions on a pointed category $\C$
are equivalent:
\begin{enumerate}[(a)]
\item The category of morphisms of $\C$ admits generic split extensions;
\item Each functor category of $\C$ with finite domain category admits generic split extensions;
\item $\C$ admits normalizers and generic split extensions;
\item The kernel functor from the category split extensions in $\C$ to $\C$
is a prefibration.
\end{enumerate}
In addition we show that the category of morphisms of a pointed protomodular $\C$ admits generic
split extensions if and only if for each morphism $f:X\to Z$ the functor sending
each object $B$ to the isomorphism class of split extensions in $\C^{\two}$
of $(B,B,1_B)$ with kernel $(X,Z,f)$, is representable.
\end{abstract}
\section{Introduction}
This paper is the third in a series of papers by the current author exploring
the notion of action representability, introduced and studied in \cite{BORCEUX_JANELIDZE_KELLY:2005a}.
In the first paper of the series \cite{GRAY:2013b}, one of the main results
was to show that for a semi-abelian category $\C$, the category
of morphisms in $\C$ is action representable if and only if $\C$ is action representable and
admits normalizers. In the second paper \cite{GRAY:2015a} it was shown that these conditions (in the semi-abelian
context) are equivalent to the functor assigning to each split extension in $\C$ its kernel being
a prefibration. The proof
from the first paper was also generalized and clarified, in the second paper, by showing that
these aforementioned equivalent conditions are in fact special cases of equivalent
conditions about an abstract functor together with a monad on its domain. In this paper
we follow the approach of the second paper, generalizing and simplifying certain results and adding others.

For a pointed category $\C$, a split extension (of $B$ with kernel $X$) is a diagram in $\C$
\begin{equation}
\label{equation: a split extension}
\xymatrix{X \ar[r]^{\kappa} & A \ar@<0.5ex>[r]^{\alpha} & B\ar@<0.5ex>[l]^{\beta}}
\end{equation}
where $\kappa$ is the kernel of $\alpha$, and $\alpha\beta = 1_B$.  A morphism of split extensions is a diagram in $\C$

\begin{equation}
\label{equation: morphism of split ext}
\vcenter{
\xymatrix{
X \ar[r]^{\kappa} \ar[d]_{u}& A \ar@<0.5ex>[r]^{\alpha}\ar[d]^{v} & B\ar@<0.5ex>[l]^{\beta}\ar[d]^{w}\\
X' \ar[r]^{\kappa'} & A' \ar@<0.5ex>[r]^{\alpha'} & B'\ar@<0.5ex>[l]^{\beta'}
}
}
\end{equation}
where the top and bottom rows are split extensions (the domain and codomain respectively), and $v\kappa = \kappa' u$, $v\beta = \beta' w$ and $w\alpha=\alpha'v$. The split short five lemma holds in a
pointed category $\C$ if for each diagram \eqref{equation: morphism of split ext} in $\C$, the
morphism $v$ is an isomorphism as soon as both $u$ and $w$ are.
A pointed finitely complete category is (Bourn-)protomodular \cite{BOURN:1991}
if it satisfies the split short five lemma and is semi-abelian \cite{JANELIDZE_MARKI_THOLEN:2002} if in addition it is exact and admits binary coproducts. 
Writing $\SE{\C}$ for the category of split extensions we can construct the span 
\begin{equation}
\label{split ext span}
\xymatrix{
\C & \SE{\C} \ar[l]_-{P} \ar[r]^-{K} & \C
}
\end{equation}
where $P$ and $K$ are the functors sending the split extension \eqref{equation: a split extension} to $B$ and $X$, respectively. A generic split extension with kernel $X$ (in the sense
of \cite{BORCEUX_JANELIDZE_KELLY:2005a}) is a terminal object in the category $K^{-1}(X)$. We denote such a split extension as follows
\begin{equation}
\label{equation: generic split extension}
	\xymatrix{X \ar[r]^-{k_X} & H(X) \ar@<0.5ex>[r]^-{p_X} & \spexc{X}. \ar@<0.5ex>[l]^-{i_X}}
\end{equation}
If $\C$ is pointed protomodular and finitely complete, then each object $\spexc{X}$ is called a split extension classifier
and is the representing object of a functor 
\begin{equation}
\label{equation: splext functor}
\SplExt(-,X) : \C^{\op} \to \Set
\end{equation}
which
sends an object $B$ to the isomorphism class of split extensions of the form \eqref{equation: a split extension}, and which is defined on morphisms by \emph{pullback}. When $\C$ is semi-abelian
there is an equivalence of categories between the category of $\SE{\C}$ and a category of
\emph{internal object actions} (see \cite{BORCEUX_JANELIDZE_KELLY:2005a}). The functors
\eqref{equation: splext functor} turn out to be isomorphic to ones defined in terms of internal
object actions, and the category $\C$ is called action representable when each of these
functors is representable.
We will also write $\FSE{X}{\C}$
for the category $K^{-1}(X)$. We write $\C^{\two}$ for the category of morphisms in $\C$
and denote its objects as triples $(X,Z,f)$ where $X$ and $Z$ are objects and $f:X\to Z$ is
a morphism in $\C$.

Action representable categories include:
\begin{itemize}
\item the category of groups where $\spexc{X}=\Aut(X)$ is the
	group of automorphisms of $X$;
\item the category of Lie algebras where $\spexc{X}=\Der(X)$ is
	the Lie algebra of derivations of $X$;
\item the category of not necessarily unitary Boolean rings where $\spexc{X}=\End(X)$ is
	the Boolean ring of linear endomorphisms of the underlying abelian
	groups of $X$;
\end{itemize}
The main aim of the paper is threefold: 
(i) to show Theorem 4.8 of \cite{GRAY:2013b} generalizes to the pointed protomodular context with a simpler proof;
(ii) to clarify the role of the objects $\spexc{X,Z,f}$ for each morphism $f:X\to Z$ which arise as part of the $\spexc{(X,Z,f)}$
in the category of morphisms of $\C^\two$; 
(iii) to produce a simple construction of $\spexc{X,Z,f}$ in terms of normalizers.

For $f:Z\to Z$ in the category of groups $\spexc{X,Z,f}$ is the group of pairs of automorphisms of $X$ and $Z$,
respectively, compatible with $f$, that is the subgroup 
\[
	\{(\theta, \phi)\in \Aut(X)\times \Aut(Z)\mid f \circ \theta = \phi \circ f\}
\]
of $\Aut(X)\times \Aut(Z)$, which as a special case of Proposition \ref{2d is normalizer pullback} can
be constructed as
\[
(i_X\times i_Z)^{-1}(N_{H(X)\times H(Z)}(\langle k_X, k_zf \rangle(X)))
\]

\section{Spans, pre-fibrations and terminal objects in fibers}
For a functor $F:\A\to \X$ and an object $B$ in $\X$ we write $F_B$
for the functor $(\A\downarrow B) \to (\X\downarrow F(B))$ which
sends an object $(A,\alpha)$ to $(F(A),F(\alpha))$ and sends
a morphism $f:(A,\alpha)\to (A',\alpha')$ to $F(f)$.
\begin{lemma}\label{product_lemma}
Let $\A$ be a category with binary products, $F:\A\to \X$ a functor which preserves binary
products, $B$ an object in $\A$ and $X$ be an object in $\X$.
If $T$ is a terminal object in $F^{-1}(X)$, then for each $\theta:X\to F(B)$ in $\X$
the functor
$F_{T\times B}^{-1}(X,\langle 1,\theta\rangle) \to F_B^{-1}(X,\theta)$
defined on objects by $(A,\gamma)$ to $(A,\pi_2\gamma)$ and as identity on morphisms, is an isomorphism.
\end{lemma}
\begin{proof}
Let $\theta: X\to Y$ be an object in $\X$.
For an object $(A,\alpha)$
in $F_B^{-1}(X,\theta)$ let $u:A\to T$ be the unique morphism in $F^{-1}(X)$.
It is easy to check that the pair $(A,\langle u,\alpha\rangle)$ 
is in
$F_{T\times B}^{-1}(X,\langle 1,\theta\rangle)$, and this assignment
produces the desired inverse functor.
\end{proof}
Recall that a functor weakly creates limits when each limiting cone of
the image of a diagram under the functor is the image of a limiting cone
under the functor.  
\begin{proposition}\label{prefib}
Let $F:\A\to \X$ be a functor which weakly creates limits, and suppose that $\X$ has finite products.
The conditions are equivalent:
\begin{enumerate}[(a)]
\item $F$ is prefibration;
\item $F$ admits $F$-precartesian liftings of monomorphisms and has terminal objects in its fibers;
\item $F^\two$ has terminal objects in its fibers.
\end{enumerate}
\end{proposition}
\begin{proof}
The equivalence of (a) and (c) was established in Theorem 2.24 of \cite{GRAY:2017}.
Since the same theorem tells us that (a) implies that $F$ admits terminal objects in its fibers, 
we see that (a) implies (b). The converse follows from the previous lemma.
\end{proof}
\begin{remark}\label{precarteian_above_mono}
Note that if $F:\A\to \X$ is a functor which weakly creates limits, then each precartesian
morphism above a monomorphism is a monomorphism. Indeed if $f:A\to B$ is such an $F$-precartesian
morphism and $F(f)$ is a monomorphism, then by assumption the kernel pair $k_1,k_2:K\to A$
of $f$ exists and can be chosen to be above the kernel pair $1_{F(A)},1_{F(A)}:F(A)\to F(A)$
of $F(f)$. It immediately follows that $k_1=k_2$ and hence $f$ is a monomorphism.
\end{remark}
\begin{lemma}\label{fiberous adunctions}
For a diagram of adjunctions
\[
\xymatrix{
\S \ar@<-0.75ex>[d]_{I}^{\!\!\dashv} \ar[r]^{Q} & \B\ar@<-0.75ex>[d]_{J}^{\!\!\dashv}\\
\S' \ar@<-0.75ex>@{->}[u]_{R} \ar[r]^{Q'} & \B'\ar@<-0.75ex>@{->}[u]_{S}
}
\]
with units and counits $\eta: 1_{\S}\to RI$, $\epsilon : IR\to 1_{\S'}$, $\phi:1_{\B}\to SJ$ and $\theta : JS\to 1_{\B'}$. 
If the downward directed and upward directed arrows commute
(that is, $Q'I =JQ$ and $QR=SQ'$), then $Q'\circ \epsilon = \theta \circ Q'$ if and only if $Q \circ \eta = \phi Q$.
When this is the case for each $B$ in $\B$ the adjunction $I \dashv R$
restricts along the inclusions $Q^{-1}(B)\to \S$ and $Q'^{-1}(J(B)) \to \S'$ to
an adjunction $I_B \dashv R_B$ as soon as $SJ=1_{\B}$ and $\phi = 1_{1_\B}$.
\end{lemma}
\begin{proof}
Suppose that $Q'\circ \epsilon = \theta \circ Q'$. Let $S$ be an element of $\S$.
Noting that $JQRI(S)=JSJQ(S)$ we have $\theta_{Q'I(s)}JQ(\eta_S)=Q'(\epsilon_{I(S)})Q'(I(\eta_S))=Q'(1_{I(S)})=1_{Q(S)}$,
and hence by the universal property of $\theta$ we have $Q(\eta_S)=\phi_{Q(S)}$. Therefore $Q\circ \eta = \phi \circ Q$.
The converse is dual. Now suppose that we have $Q'\circ \epsilon = \theta \circ Q'$, $Q\circ \eta = \phi \circ Q$, $SJ=1_{\B}$
and $\phi=1_{1_\B}$.
Suppose that $f:S\to T$ is in $Q^{-1}(B)$ and $g:S'\to T'$ is in $Q^{-1}(J(B))$. We have $Q'I(f)=JQ(f)=J(1_B)=1_{J(B)}$,
$QR(g)=SQ'(g)=S(1_{J(B)})=1_{SJ(B)}=1_B$, $Q'(\epsilon_S')=\theta_{Q'(S')}=\theta_{J(B)}=J(1_{B}))\theta_{J(B)}=J(\phi_B)\theta_{J(B)}=1_{J(B)}$,
and $Q(\eta_S)=\phi_{Q(S)}=\phi_B=1_B$. It follows that the adjunction restricts as desired.
\end{proof}
Let $\bS=$
\[
\xymatrix{
\A & \S \ar[l]_{P} \ar[r]^{Q} & \B
}
\]
be a span of categories. Consider the following diagram of spans
\[
\xymatrix{
\A \ar[d]_{\Delta}& \underline \S^\two \ar[d]^{I} \ar[l]_{\underline P^\two} \ar[r]^{\underline Q^\two} & \B^\two\ar@{=}[d]\\
\A^\two & \S^\two \ar[l]_{P^\two} \ar[r]^{Q^\two} & \B^\two
}
\]
where the left hand square is a pullback. Let us choose $\underline \S^{\two}$ so that
it is the full subcategory $\S^{\two}$ with objects those objects $(S,S',\sigma)$ such that $P(S)=P(S')$ and $P(\sigma)=1_{P(S)}$.
Let us write $\underline \bS^\two$ and $\bS^\two$
for the span at the top and bottom of the previous diagram.

It is easy to check that $\textnormal{dom} \dashv \Delta : \S \to \S^{\two}$ restricts to
$\underline {\textnormal{dom}} \dashv \underline \Delta : \S \to \underline\S^{\two}$ and furthermore the diagram
\[
\xymatrix{
\S \ar@<-0.75ex>[d]_{\underline \Delta}^{\!\!\dashv} \ar[r]^{Q} & \B\ar@<-0.75ex>[d]_{\Delta}^{\!\!\dashv}\\
\underline{\S^{\two}} \ar@<-0.75ex>@{->}[u]_{\underline{\textnormal{dom}}} \ar[r]^{\underline Q^{\two}} & \B^{\two}\ar@<-0.75ex>@{->}[u]_{\textnormal{dom}}
}
\]
satisfies the requirements of Lemma \ref{fiberous adunctions} so that for each $B$ there is an induced adjunction
\[
\xymatrix{
Q^{-1}(B)\ar@<1.2ex>[r]^-{\underline{\Delta}_B}_-{\bot} & \underline {Q^{\two}}^{-1}(B,B,1_B)\ar@<1.2ex>[l]^-{\underline {\textnormal{dom}}_B}
}
\]
We obtain:
\begin{lemma}\label{two_to_one}
If $\underline {Q^{\two}}$ has terminal objects in its fibers, then so does $Q$.\qed
\end{lemma}
Recall that a span $\bS$ is right regular \cite{JANELIDZE:1978} (which was first considered as part of Yoneda's definition of regular span \cite{YONEDA:1961}) if for each object $S$ in $\S$ and  each morphism $f : A\to P(S)$ in $\A$ there
exists a $P$-cartesian morphism $\sigma : \bar S\to S$ above $f$ such that $Q(\sigma) = 1_{Q(S)}$. 
\begin{lemma}
If $\bS$ is right regular, then the functor $I$ has a right adjoint $R$ such that $\underline Q^\two\circ R = Q^\two$ and $\underline P^\two \circ R = \textnormal{dom}\circ P^\two$.
\end{lemma}
\begin{proof}
Let $(S,S',\sigma)$ be an object in $\S^\two$ and write $p:\tilde S \to S'$ for the $P$-cartesian
lifting of $P(\sigma)$ to $S'$ such that $Q(p) = 1_{Q(S')}$. It follows that there is a unique
morphism $\tilde \sigma : S\to \tilde S$ such that $\sigma = p \tilde \sigma$ and $P(\tilde \sigma)=1_{P(S)}$. We define $R(S,S',\sigma)=(S,\tilde S,\tilde \sigma)$. We will show that the morphism forming the right hand square in the diagram (ignore the left hand part and outer arrows)
\[
\xymatrix{
T \ar@/^3ex/[rr]^{u}\ar[d]_{\tau} \ar@{-->}[r]_{u}& S \ar[r]_{1_S}\ar[d]_{\tilde \sigma} & S\ar[d]^{\sigma}\\
T' \ar@/_3ex/[rr]_{v}\ar@{-->}[r]^{v'}& \tilde S \ar[r]^{p} & S'
}
\]
is the counit of the adjunction $I\dashv R$. Suppose that $(u,v)$ as displayed above is a
morphism from $I(T,T',\tau)=(T,T',\tau)$ to $(S,S',\sigma)$. Since $P(v)= P(v)P(\tau) = P(\sigma)P(u)$
it follows that there is a unique morphism $v':T'\to \tilde S$ such that $p v' = v$ and $P(v')=P(u)$.
An easy calculation using the universal property of $p$ shows that $v' \tau = \tilde \sigma u$ producing
a morphism $(u,v'):(T,T',\tau)\to (S,\tilde S, \tilde \sigma)$ which when composed with $(1_S,p)$
gives $(u,v)$. Uniqueness of a morphism with this property follows easily from the universal property of $p$. To show that $\underline Q^\two \circ R = Q^{\two}$ just apply $Q$ to the right hand square of the above diagram. It is immediate that $\underline P^\two \circ R = \textnormal{dom}\circ P^\two$.
\end{proof}
\begin{proposition}\label{span:right_regular_reform}
Suppose $\bar S$ is right regular.
The functor $R$, from above, 
maps terminal objects in fibers of $Q^\two$ to terminal
objects in the fibers of $\underline Q^\two$. The functor $\underline Q^\two$ has terminal
objects in its fibers if and only if $Q^\two$ does.
\end{proposition}
\begin{proof}
The previous lemma together with Lemma \ref{fiberous adunctions} implies that $R$ preserves terminal objects in fibers and hence
$\underline Q^\two$ has terminal objects in its fibers as soon as $Q^\two$ does.  Now
suppose that $\underline Q^\two$ has terminal objects in its fibers. Lemma \ref{two_to_one} tells us that $Q$
has terminal objects in its fibers.

Now suppose that $\beta : B\to B'$ is a morphism in $\B$ and let
$(\bar S, \bar S', \bar \sigma)$ be a terminal object in $\underline Q^{\two^{-1}}(B,B',\beta)$.
Let $\bar T$ be a terminal object in $Q^{-1}(B')$ and let $p:\bar S' \to \bar T$ be
the unique morphism into the terminal object in $Q^{-1}(B')$. Let us observe
that $p$ is $P$-cartesian. To that end suppose
that $q: T'\to \bar T$ is a $P$-cartesian lifting of $P(p)$ to $\bar T$ such
that $Q(q)=1_{Q(\bar T)}$. It follows that there is a unique morphism $k:\bar S'\to T'$
such that $P(k)=1_{P(T')}$. We obtain the diagram 
\[
\xymatrix{
\bar S \ar[d]_{\bar \sigma} \ar[r]^{1_{\bar S}} & \bar S \ar[d]^{k\bar \sigma} \ar[r]^{l} & \bar S \ar[d]^{\bar \sigma}\\
\bar S' \ar[r]_{k} & T' \ar[r]_{m} & \bar S'
}
\]
in 
$\underline Q^{\two^{-1}}(B,B',\beta)$ where the right hand morphism is
the unique morphism into the terminal object. The universal property tells us
that $l1_{\bar S}=1_{\bar S}$ and $mk=1_{\bar S'}$ and hence $l=1_{\bar S}$.
Since $P(km)=1_{P(\bar S')}$ it follows
that $km=1_T'$ (via the universal property of $q$). It now follows that $p$ is
$P$-cartesian. We will now show that
$(\bar S,\bar T, p\bar \sigma)$ is a terminal object in $Q^{\two^{-1}}(B,B',\beta)$.
Let $(S,T,g)$ be an object $Q^{\two^{-1}}(B,B',\beta)$ and let $r:\tilde S\to T$
be a $P$-cartesian lifting of $P(g)$ to $T$ such that $Q(r)=1_{Q(T)}$.
Let $\sigma : S\to \tilde S$ be the unique morphism such that $r\sigma = g$
and  $P(\sigma)=1_{P(S)}$. We have that $Q(\sigma) = Q(r)Q(\sigma)=Q(g)=\beta$.
Therefore the universal property of $(\bar S, \bar S',\bar \sigma)$ produces
a unique morphism $(u,v):(S,\tilde S, \sigma) \to (\bar S, \bar S',\bar \sigma)$.
On the other hand we have a unique morphism $w: T\to \bar T$. Since $\bar T$ is terminal,
it follows that $pv=wr$ and hence we obtain a
morphism $(u,w):(S,T,g)\to (\bar S,\bar T,p\bar\sigma)$ in $Q^{\two^{-1}}(B,B',\beta)$.
Suppose $(u',w')$ is another such morphism. The universal property of
$\bar T$ tells us that $w'=w$. Since $p$ is $P$-cartesian and
$p v \sigma = w r\sigma = w g =  p\bar \sigma u'$ and $P(v \sigma) = P(\bar \sigma u') = 1_{P(\bar S)}$
it follows that $v \sigma = \bar \sigma u'$ and hence
$(u',v)$ is a morphism from $(S,\tilde S,\sigma)$
to $(\bar S, \bar S',\bar \sigma)$ and hence must be $(u,v)$.
\end{proof}

\section{Action representability and normalizers}
In this section we explore the consequences of the results from
the previous section in the case where $\bS$ is the span displayed in \eqref{split ext span}.

Before doing so we need to recall a few definitions. In \cite{GRAY:2013b} the normalizer of a monomorphism $f:X'\to X$ in a pointed category $\C$ was defined as the terminal object
in the category of factorizations of $f$ of the form $(N,m,n)$ where $f=mn$, $m$ is a monomorphism, and $n$ is a normal monomorphism.
In \cite{BOURN_GRAY:2015} a definition of the normalizer of a monomorphism in an arbitrary category $\C$ was given, which is different to the definition in \cite{GRAY:2013b} even when $\C$ is pointed. Indeed, when $\C$ is pointed and has finite limits the normalizer of a monomorphism $f:X'\to X$ in the sense of \cite{BOURN_GRAY:2015} can equivalently be defined as a commutative diagram
\[
\xymatrix{
X' \ar@/_2ex/[dd]_{f} \ar[d]^{n} \ar[r]^{\kappa} & R \ar[d]^{\langle r_1,r_2\rangle}\\
N\ar[r]^-{\langle 0,1\rangle}\ar[d]^{m} & N\times N\\
X
}
\]
where the upper square is a pullback and the morphisms $r_1,r_2 : R\to N$ are the projections of an equivalence relation, which is universal amongst such commutative diagrams. It can be seen that the morphisms $m$ and $n$ are necessarily a monomorphism (see Proposition 2.3 \cite{BOURN_GRAY:2015})  and a Bourn-normal monomorphism, respectively, and furthermore that $n$ is a normal monomorphism (that is, the kernel of some morphism) when $\C$ is exact. The two definitions are equivalent when there is a unique effective equivalence relation to which each normal monomorphism is normal, and the image of a normal monomorphism along a regular epimorphism is normal, which is the case in the semi-abelian context.
In the protomodular context the relation $R$ in the above diagram is unique up to isomorphism and hence the normalizer in the sense of \cite{BOURN_GRAY:2015}
becomes the terminal factorization $f=mn$ where $n$ is a Bourn-monomorphism.

We have:
\begin{theorem}\label{main1}
For a pointed finitely complete category $\C$ the following are equivalent.
\begin{enumerate}[(a)]
\item $\C$ has generic split extensions and admits normalizers in the sense of \cite{BOURN_GRAY:2015};
\item the category $\C^\two$ has generic split extensions;
\item the functor $K$ is a prefibration.
\end{enumerate}
\end{theorem}
\begin{proof}
By Theorem 2.8 of \cite{BOURN_GRAY:2015} we know that the existence of normalizers
in the sense of \cite{BOURN_GRAY:2015} is equivalent to the functor $K$ admitting $K$-precartesian
liftings of monomorphisms. The proof of the claim now follows from Proposition \ref{prefib}.
\end{proof}
Let us explain explicitly how normalizers and generic split extensions in a pointed finitely
complete category $\C$ give rise to generic split extensions in $\C^{\two}$. To that end
suppose that $f:X\to Z$ is a morphism in $\C$. By assumption the normalizer
\[
\xymatrix{
X \ar@/_2ex/[dd]_{\langle k_X,k_Zf\rangle} \ar[d]^{n} \ar[r]^{\kappa} & R \ar[d]^{\langle r_1,r_2\rangle}\\
N\ar[r]^-{\langle 0,1\rangle}\ar[d]^{m} & N\times N\\
H(X)\times H(Z)
}
\]
exists. Consider the following pullback in $\SE{\C}$ 
 \[
\xymatrix@!@C=-20ex@R=-26ex{
 X
  \ar[dd]_{\langle 1,f\rangle}
  \ar[rr]^-{k_f}
  \ar@{=}[rd]&&
{H(A,Z,f)}
 \ar[dd]^(0.30){{j}}|\hole
\ar@<2.00pt>[rr]^{{p_f}}
\ar[rd]^{{v}}&&
{\spexc{X,Z,f}}
 \ar[dd]^(0.3){q}|(0.48)\hole|(0.51)\hole
 \ar@<2.00pt>[ll]^(0.4){{i_f}}
 \ar[rd]^{u}&\\
&X\ar[rr]^(0.30){\kappa}\ar[dd]^(0.25){\langle k_X,k_Zf\rangle}&&
R\ar[dd]^(0.3){\langle mr_1,mr_2\rangle}\ar@<2.00pt>[rr]^(0.30){r_1}&&
N\ar[dd]^{m}\ar@<2.00pt>[ll]^(0.70){s}\\
X\times Z\ar[rr]^(0.30){k_X\times k_Z}|\hole\ar[rd]_{k_X\times k_Z}&&
H(X)\times H(Z)\ar@<2.00pt>[rr]^(0.35){p_X\times p_Z}|\hole\ar[rd]_(0.4){\langle i_Xp_X\times i_Zp_Z,1\rangle\,\,\,}&&
\spexc{X}\times\spexc{Z}\ar@<2.00pt>[ll]|\hole^(0.65){i_X\times i_Z}\ar[rd]_{i_X\times i_Z}&\\
&H(X)\times H(Z)
\ar[rr]_-{\langle 0,1\rangle}&&
(H(X)\times H(Z)\times (H(X)\times H(Z))
\ar@<2.00pt>[rr]^-{\pi_1}&&
H(X)\times H(Z)
\ar@<2.00pt>[ll]^-{\langle 1, 1 \rangle}
}
 \]
The proof of Proposition 2.4 together with Proposition 2.6 of \cite{BOURN_GRAY:2015} tell us that the morphisms forming the front
and back faces of the previous diagram are $K$-precartesian. Lemma \ref{product_lemma} now tells us that
the morphism 
\[
\xymatrix{
 X
  \ar[d]_{f}
  \ar[r]^-{k_f}
  &
{H(A,Z,f)}
 \ar[d]^{\pi_2j}
\ar@<2.00pt>[r]^{p_f}
&
\spexc{X,Z,f}
 \ar[d]^{\pi_2 q}
 \ar@<2.00pt>[l]^{i_f}
 \\
Z\ar[r]^(0.30){k_Z}&
H(Z)\ar@<2.00pt>[r]^(0.35){p_Z}&
\spexc{Z}\ar@<2.00pt>[l]^{i_Z}
}
\]
is $K$-cartesian. Finally Proposition 2.6 of \cite{GRAY:2017} tells us that the previous
diagram is the generic split extension with kernel $(X,Z,f)$ in $\C^{\two}$.
It seems worthwhile also mentioning that $j$ and $q$ are necessarily monomorphisms
via Remark \ref{precarteian_above_mono}.

Let us write $\TFSE{(X,Z,f)}{\C}$ for the category $\underline K^{\two^{-1}}(X,Z,f)$ which has
objects of the forming
\begin{equation}
\label{mor_split_ext_common_cod}
\vcenter{
\xymatrix{
	X \ar[r]^{\kappa} \ar[d]_{f}& A \ar@<0.5ex>[r]^{\alpha}\ar[d]^{g} & B\ar@<0.5ex>[l]^{\beta}\ar@{=}[d]\\
Z \ar[r]^{\sigma} & C \ar@<0.5ex>[r]^{\gamma} & B.\ar@<0.5ex>[l]^{\delta}
}
}
\end{equation}
Via Proposition \ref{span:right_regular_reform} we obtain:
\begin{proposition}
Let $\C$ be a pointed category.
The category $\C^{\two}$ has generic split extensions if and only if
for each morphism $f:X\to Z$ the category $\TFSE{(X,Z,f)}{\C}$ 	
has a terminal object.\qed
\end{proposition}
In the presence of protomodularity the existence of terminal objects
in $\TFSE{(X,Z,f)}{\C}$ is equivalent to representability of 
\[
\SplExt_{\C^{\two}}(\Delta(-),(X,Z,f)): \C^{\op} \to \Set.
\]
When the above functor is representable one easily sees that $\spexc{X,Z,f}$ is its
representing object.
We obtain:
\begin{theorem}
A pointed finitely complete protomodular category is action representable with normalizers
in the sense of \cite{BOURN_GRAY:2015}
if and only if for each $f:X\to Z$ the functor 
\[
\SplExt_{\C^{\two}}(\Delta(-),(X,Z,f)): \C^{\op} \to \Set.
\]
is representable \qed
\end{theorem}

Via the construction immediately following Theorem \ref{main1} we obtain:
\begin{proposition}\label{2d is normalizer pullback} Let $\C$ be a pointed protomodular category admitting
normalizers and let $f:X\to Z$ be a morphism in $\C$. If $X$ and $Z$
  admit generic split extensions, then the object $\spexc{X,Z,f}$ exists and forms
  part of a pullback $(\spexc{X,Z,f},u,q)$ in the diagram
\[
  \xymatrix{
    X\ar[d]_{\langle 1,f \rangle} \ar[r]^{n} & N \ar[d]^{m} & \spexc{X,Z,f} \ar[l]_{u}\ar[d]^{q} \\
    X\times Z \ar[r]_-{k_X\times k_Z} & H(X) \times H(Z) &  \spexc{X}\times \spexc{Z} \ar[l]^-{i_X\times i_Z}
  }
\]
  in which $(N,n,m)$ is the normalizer of $\langle k_X,k_Zf \rangle$.\qed
\end{proposition}
For $B$ in $\C$ let us write $K_B$, for the composite of the 
inclusion $P^{-1}(B)\to \SplExt(\C)$ and $K$.
Since protomodularity is equivalent to each functor $K_B$ reflecting isomorphisms,
and since each $K_B$ functor preserves limits it follows that each $K_B$ functor is faithful.
We obtain:
\begin{lemma}
Let $\C$ be a pointed finitely complete protomodular category.
Given a morphism $f:X\to Z$ and split extensions as shown
at the top and bottom of the diagram
\[
\xymatrix{
	X \ar[r]^{\kappa} \ar[d]_{f}& A \ar@<0.5ex>[r]^{\alpha}\ar@{-->}[d]^{g} & B\ar@<0.5ex>[l]^{\beta}\ar@{=}[d]\\
Z \ar[r]^{\sigma} & C \ar@<0.5ex>[r]^{\gamma} & B,\ar@<0.5ex>[l]^{\delta}
}
\]
in $\C$,
there is at most one morphism $g:A\to C$ making the diagram into a morphism
of split extensions.\qed
\end{lemma}
The previous lemma easily implies that:
\begin{proposition}
Let $\C$ be a pointed protomodular category and let $f:X\to Z$
be a morphism in $\C$. The map 
\[\gamma_{B}:\SplExt_{\C^{\two}}(\Delta(B),(X,Z,f)) \to \SplExt_{\C}(B,X)\times \SplExt_{\C}(B,Z)\]
assigning to each isomorphism class
of morphism of split extensions of the form \eqref{mor_split_ext_common_cod}
to the pair isomorphism classes of its domain and codomain, respectively,
is the component of a natural monomorphism.
\end{proposition}
For a morphism $f:X\to Z$ such that
	$\spexc{X}$, $\spexc{Z}$ and
   $\spexc{X,Z,f}$ exist, the morphism $q : \spexc{X,Z,f}\to\spexc{X}\times \spexc{Z}$
   defined above makes the diagram
\[
\xymatrix{
	\SplExt_{\C^{\two}}(\Delta(-),(X,Z,f)) \ar[r]^-{\gamma}\ar[d] & \SplExt_{\C}(-,X)\times \SplExt_{\C}(-,Z)\ar[d]\\
	\hom(-,\spexc{X,Z,f}) \ar[r]^{\hom(-,q)} & \hom(-,\spexc{X}\times\spexc{Z}),
}
\]
in which the vertical unlabelled morphisms are the natural isomorphisms forming
part of the respective representations, commute.
As an immediate corollary of the previous proposition we obtain:
\begin{corollary}
  Let $\C$ be a pointed protomodular category and let $f:X\to Z$
  be a morphism in $\C$, such that
	$\spexc{X}$, $\spexc{Z}$ and
   $\spexc{X,Z,f}$ exist. Given a diagram with
  solid arrows
  \begin{equation}\label{split_ext_to fill}
    \vcenter{
    \xymatrix{
      X
      \ar[r]^{\kappa}
      \ar[d]_{f} 
      &
      A
      \ar@{-->}[d]^{g}
      \ar@<0.5ex>[r]^-{\alpha}
      &
      B
      \ar@<0.5ex>[l]^-{\beta}
      \ar@{=}[d]
      \\
      Z
      \ar[r]^{\sigma}
      &
      A
      \ar@<0.5ex>[r]^-{\gamma}
      &
      B
      \ar@<0.5ex>[l]^-{\delta}
    }
  }
  \end{equation}
  there exists a morphism $g$ making the diagram into a
  morphism of split extensions if and only if the 
  morphism $\langle \theta,\varphi\rangle :B\to \spexc{X}\times \spexc{Z}$
  induced by the morphisms $\theta : B\to \spexc{X}$ and
  $\varphi: B\to \spexc{Z}$ which determine the split
  extensions at the top and bottom of the above diagram, respectively,
  factors through $q : \spexc{X,Z,f} \to \spexc{X}\times \spexc{Z}$.
\end{corollary}
\providecommand{\bysame}{\leavevmode\hbox to3em{\hrulefill}\thinspace}
\providecommand{\MR}{\relax\ifhmode\unskip\space\fi MR }
\providecommand{\MRhref}[2]{%
  \href{http://www.ams.org/mathscinet-getitem?mr=#1}{#2}
}
\providecommand{\href}[2]{#2}

\end{document}